\documentclass[11pt]{amsart}
\usepackage[margin=1in]{geometry}
\usepackage{amsmath,amssymb,amsthm,amsfonts,amsxtra,xspace,mathtools}
\usepackage{enumerate,verbatim,mathrsfs,comment,color,url}
\usepackage[all,2cell,ps]{xy}
\usepackage[pagebackref, colorlinks=true, pdfstartview=FitV, linkcolor=red, citecolor= blue , urlcolor=blue]{hyperref}
\usepackage{graphicx}
\usepackage{tikz-cd}
\usepackage{verbatim}
\usepackage{stackengine}
\usepackage{cleveref}
\usepackage{parskip}

\DeclareMathOperator{\depth}{depth}

\DeclareMathOperator{\embdim}{embdim}

\DeclareMathOperator{\Hom}{Hom}

\DeclareMathOperator{\id}{id}
\DeclareMathOperator{\im}{im}

\DeclareMathOperator{\projdim}{projdim}

\DeclareMathOperator{\rank}{rank}

\DeclareMathOperator{\tr}{tr}

\renewcommand{\ge}{\geqslant}
\renewcommand{\le}{\leqslant}

\newcommand{\sk}{\mbox{$\mathsf{k}$}}
\newcommand{\m}{\mbox{$\mathfrak m$}}	
\newcommand{\ch}{\mbox{\rm char}}
\newcommand{\soc}{\mbox{\rm soc}}      
\newcommand{\edim}{\mbox{\rm embdim}}

\theoremstyle{plain}
\newtheorem{theorem}{Theorem}[section]
\newtheorem{lemma}[theorem]{Lemma}
\newtheorem{construction}[theorem]{Construction}
\newtheorem{proposition}[theorem]{Proposition}
\newtheorem{corollary}[theorem]{Corollary}
\newtheorem{claim}[theorem]{Claim}
\theoremstyle{definition}
\newtheorem{remark}[theorem]{Remark}

\newtheorem{innercustomtheorem}{Theorem}
\newenvironment{customtheorem}[1]
{\renewcommand\theinnercustomtheorem{#1}\innercustomtheorem}
{\endinnercustomtheorem}

\newtheorem{innercustomcorollary}{Corollary}
\newenvironment{customcorollary}[1]
{\renewcommand\theinnercustomcorollary{#1}\innercustomcorollary}
{\endinnercustomcorollary}

\theoremstyle{definition}
\newtheorem{definition}[theorem]{Definition}

\newtheorem{example}[theorem]{Example}

\newtheorem{notation}[theorem]{Notation}

\newtheorem{question}[theorem]{Question}

\numberwithin{equation}{section}

\title[
Extremal Behavior of ideals of minors]{
Extremal Behavior of ideals of minors
}

\author{Trung Chau}
\address{Chennai Mathematical Institute, Siruseri, Tamilnadu~603103. India}
\email{chauchitrung1996@gmail.com}

\author{Michael DeBellevue}
\address{Mathematics Department, Syracuse University, Syracuse, NY 13244, USA}
\email{michael.debellevue@gmail.com}

\author{Souvik Dey}
\address{Department of Mathematical Sciences, 850 West Dickson Street, University of Arkansas, Fayetteville, Arkansas 72701, USA, https://orcid.org/0000-0001-8265-3301}
\email{souvikd@uark.edu}

\author{Omkar Javadekar}
\address{Chennai Mathematical Institute, Siruseri, Tamilnadu~603103. India}
\email{omkarj@cmi.ac.in}

\author{Ganapathy Krishnamoorthy}
\address{Department of Mathematics, I.I.T. Madras, Chennai~600036, Tamilnadu. India}
\email{ganapathy.math@gmail.com}

\newcommand{\p}{\mathfrak{p}}
\newcommand{\q}{\mathfrak{q}}

\newcommand{\del}{\partial} 

\newcommand{\syz}{\Omega}
\theoremstyle{plain}

\keywords{minors,  periodicity, stretched Gorenstein rings, fiber products, minimal free resolutions, Eisenbud–Shamash construction, deformation, trace ideal}

\subjclass[2020]{13D02; 13H10; 13D10}
\begin{document}

\begin{abstract}
Ideals of minors arising from minimal free resolutions—equivalently, the Fitting ideals of syzygy modules—are natural invariants of interest in commutative algebra and algebraic geometry. A surprising observation in recent years by Brown–Dao–Sridhar is that for many nice classes of rings, like complete intersection rings and Golod rings, ideals of minors tend to become eventually periodic. In this article, we establish similar periodicity phenomena for further classes of rings, namely fiber products and Artinian stretched Gorenstein rings. In fact, we show that when the embedding dimension is at least 3, with a mild characteristic assumption, the ideals of minors stabilize to powers of the maximal ideal, exhibiting extremal behavior. We also study the transfer of periodicity between rings. Specifically, we prove that for any local ring $(R,\m)$ and a super-regular element $x\in \m$, if the ideals of minors of $R/(x)$-module $M$ have extremal behavior, then so do the ideals of minors of $M$ over $R$.
%  \\
% \textcolor{blue}{Ideals of minors arising from minimal free resolutions—equivalently, the Fitting ideals of syzygy modules—are expected to exhibit eventual periodic behavior. Recently, this was established for complete intersections and Golod rings by Brown–Dao–Sridhar. One of our main results shows that such periodic behavior transfers well under suitable deformations: for any local ring $(R,\m)$ and a super-regular element $x \in \m$, if the ideals of minors of an $R/(x)$-module M exhibit extremal asymptotic behavior, then the same holds for the ideals of minors of $M$ over $R$. We show that the extremal asymptotic behavior is exhibited by all finitely generated modules over fiber product rings and Artinian stretched Gorenstein rings with mild conditions. Consequently, our deformation results apply in these settings and yield corresponding statements for suitable higher-dimensional analogues of these rings. }
\end{abstract}
\maketitle

\section{Introduction}\label{sec:Introduction}

Let $R$ be a local (or graded) ring with maximal (homogeneous) ideal $\m$, and let $M$ be a finitely generated (graded) $R$-module.
Homological invariants of $M$ encode information about the complexity of the relations of $M$, and can be computed from its minimal free resolution.
This has been a central topic in commutative algebra and algebraic geometry.
A classical result due to Auslander, Buchsbaum, and Serre states that if $R$ is a regular local ring, then every finitely generated module has a finite minimal free resolution.
Even over such rings, constructing the minimal free resolution of a cyclic module has become an ever active topic among researchers (e.g., \cite{BPS98}).
If $R$ is not regular, minimal free resolutions are typically infinite. Thus, the quest to understand minimal free resolutions becomes much more difficult.
We refer to \cite{Avramov1998,MP15} for great surveys on the topic of infinite free resolutions. 

A common pattern in the study of homological invariants of modules is that for large homological degrees, the behavior of these invariants depends more upon the ring than upon the module.
A paradigmatic example is Eisenbud's result that minimal resolutions over hypersurface rings are eventually 2-periodic \cite{Ei80}.
This has inspired researchers to look for periodic behavior in infinite minimal free resolutions.
Eisenbud characterized the differentials in resolutions over hypersurface rings as corresponding to matrix factorizations.
In more general settings, obtaining descriptions which are similarly explicit seems far-fetched still, as it would require, as the minimum, the full understanding of all syzygies of a given module.
However, their information can be partially captured by their ideals of minors, also known as \emph{Fitting ideals} \cite{Fitting1936}, which we shall recall. 
Let $(F_n,\partial_n)$, where $n$ ranges among non-negative integers, be the minimal free resolution of $M$ over $R$.
For given integers $r$ and $n$, let $I^R_{n,r}(M)$ denote the ideal of $R$ generated by all $r\times r$ minors of a matrix that represents the differential $\partial_n\colon F_n\to F_{n-1}$.
As this definition does not depend on the bases of $F_n$ and $F_{n-1}$ upon which the matrix is constructed, the ideals of minors are invariants of the module $M$.
Applications of ideals of minors include the celebrated Buchsbaum--Eisenbud acyclicity criterion \cite{BE73} which states that whether a bounded complex of free $R$-modules is a free resolution depends entirely on the ranks of the free modules and the corresponding ideals of minors.
For more applications on ideals of minors, we refer to \cite{EG94,EH05,Wang94} for a non-exhaustive list. 

The periodicity of the differentials of resolutions over hypersurfaces implies that the ideals of minors are periodic over such rings.
Consequently, a natural question in the more general setting is whether ideals of minors are asymptotically constant, or at least periodic.
This was answered in the negative by a well-known family of examples discovered by Gasharov and Peeva \cite{GP90}. 
While the ideals of minors appearing in these examples are not periodic, their successive sum is still constant; the following weaker formulation by Dao and Eisenbud is still open in general:

\begin{question}[\protect{\cite[Question 1.1]{BDS23}}]\label{q:periodicity}
Given a positive integer $r$ and a finitely generated module $M$ over a local ring $R$, must there be a number $n(M)$ such that the sum of the ideals of $r \times r$ minors in $n(M)$ consecutive differential matrices of the minimal resolution of $M$ eventually becomes
constant?
If so, is there an integer only depending on $R$ that bounds all $n(M)$ from above?
\end{question}

The above question was answered in the affirmative by Brown, Dao, and Sridhar in the case of complete intersection and Golod rings \cite{BDS23} by using structure results for the resolutions over such rings.
Since the structure of infinite resolutions is only known in particular cases, answering \Cref{q:periodicity} in its full generality has proven elusive.

In this work we first answer \Cref{q:periodicity} in the affirmative for fiber product rings and artinian stretched Gorenstein rings.
Fiber products are a well-studied class of rings, both because they provide a means of generating new rings from old that carries over certain properties of the old rings, but also because fiber product rings typically exhibit homologically extremal behaviors. For instance, it is known from \cite[Theorem A]{quasidec} that the direct sum of the third, fourth, and fifth syzygies of any module of infinite projective dimension over a fiber product ring admits the maximal ideal as a direct summand. Similarly, stretched rings also have nice homological properties. For example, it was proved that stretched Cohen--Macaulay rings have rational Poincaré series \cite{Sa80}, and the Auslander--Reiten conjecture holds for stretched Cohen--Macaulay rings \cite{Gu20}. 

We remark that if $R$ is a noetherian local ring with $\edim(R)\le 2$, then in view of \cite[5.1, Proposition 5.3.4]{Avramov1998}, $R$ is either a complete intersection or a Golod ring. Thus, over such rings, given any finitely generated $R$-module $M$, we have $I_{n,r}^R(M) = I_{n+2, r}^R(M)$ for $n \gg 0$ by \cite[Theorems 1.2 and 1.3]{BDS23}. Therefore in our first main result below, we will only consider the case $\embdim(R)\geq 3$.

\begin{customtheorem}{A}[\protect{\Cref{t:FiberProduct}, \Cref{thm:main-theorem-stretched-Gorenstein}}]\label{introthm:periodicity}
    Let $(R,\m,\sk)$ be a noetherian local ring with $\embdim(R)\geq 3$.
    If $R$ is a fiber product ring or an artinian stretched Gorenstein ring, with $\ch(\sk)\neq 2$ in the latter case, then for any positive integer $r$ and any $R$-module $M$ where $\projdim_R M=\infty$, we have $I_{n,r}^R(M)=\m^r$ for $n\gg 0$.
\end{customtheorem}
In the case of fiber products, we obtain a bound on $n$, independent of the module $M$, after which the ideal of minors $I_{n,r}^R(M)$  stabilizes  (see \Cref{t:FiberProduct}). In contrast, over artinian stretched Gorenstein rings, we show that no such bound independent of $M$ exists (see \Cref{exam:no-uniform-bound-stretched-gorenstein}). 

Combining \Cref{introthm:periodicity} and \cite[Theorems 1.2 and 1.3]{BDS23}, we obtain the following.

\begin{customcorollary}{B}
    Let $(R,\m,\sk)$ be a noetherian local ring.
    If $R$ is a fiber product ring or an artinian stretched Gorenstein ring, with $\ch(\sk)\neq 2$ in the latter case, then for any positive $r$ and any $R$-module $M$ where $\projdim_R M=\infty$, we have $I_{n,r}^R(M)=I_{n+2,r}^R(M)$ for $n\gg 0$.
\end{customcorollary}

We can also connect the extremal behavior of ideals of minors seen in \Cref{introthm:periodicity} to the extremal behavior of trace ideals of high enough syzygies of modules. There has been extensive recent interest around trace ideals in connection with understanding centers of endomorphism rings (\cite{lindo}), rings close to being Gorenstein (\cite{nearly}), Arf rings (\cite{arf}),  module closure operations (\cite{closue}), Berger's conjecture (\cite{berger}), to name a few. As a consequence of \Cref{trace} it follows that all high enough syzygies of every module of infinite projective dimension over local rings as described in \Cref{introthm:periodicity} have trace equal to the maximal ideal.

Intuitively, when $R$ is a standard-graded ring, $I_{n,r}^R(M)=\m^r$ for $n\gg 0$ requires a sufficiently large quantity of linear entries in the matrices representing the differentials $\partial_n$ of the minimal free resolution of $M$ for large enough $n$.
In fact, our proofs rely on finding such entries, rather than describing the entire minimal free resolution.
Linearity of resolutions is a more general topic which has received extensive study. For instance, Green and Lazersfeld's $N_p$ property concerns linearity of all entries of $\del_n$ for $0\leq n\leq p$ \cite{Farkas2017}, and linearity of all entries of every $\del_n$ implies that $M$ is a Koszul module \cite{Green1997}.
A weaker notion, made formal in \cite[Definition 2.1]{DaoEisenbud2023}, is to somehow count the number of distinct linear entries appearing in a matrix.

A result on the periodicity of ideals of minors is known over rings of Burch index at least 2. For more on Burch index, we refer interested readers to \cite{DaoEisenbud2023,DKT20}. Let $(R,\m,\sk)$ be a local ring of depth $0$.
It is shown in \cite[Theorem 0.1]{DaoEisenbud2023} that if $R$ is of Burch index at least 2, then the 7-th syzygy module of any finitely generated non-free module $M$ contains $\sk$ as a direct summand.
Coupling this with \cite[Theorem~2.15]{BDS23}, given that $R$ is not a hypersurface, we have that $I_{n,r}^R(M)=\m^r$ for $n\gg 0$.
Artinian stretched Gorenstein rings are typically not hypersurface rings (see \cite{Sa79}), and consequently must be of  Burch index 0 (see \cite[Example 5.5]{DaoEisenbud2023}). 
As for fiber product rings, if both factors have depth zero the Burch index of the fiber product is the sum of the Burch indices of the factors, which produces many fiber product rings of Burch index less than two.
A particular family of examples where one of the factors does not have depth $0$ is $\sk[x]/(x^t)\times_{\mathsf k} T$ for any $t\geq 3$ and any regular ring $T$; these rings have Burch index $1$.
As \Cref{introthm:periodicity} makes no presumption on the Burch index of the ring, our results complement those known over rings of Burch index at least 2.

Investigation of \Cref{q:periodicity} has focused so far on establishing periodicity results for particular classes of rings.
These results can be extended by investigating how periodicity of ideals transfers along ring homomorphisms.
Deformations are a class of ring homomorphisms that are especially amenable to the study of transfer of properties of rings; recall that we say that a noetherian local ring  $(R,\m)$ is a deformation of $R'$ % do we say R is a deformation of R' ...OR... R' is a deformation of R...... 
if there is a homomorphism $\varphi\colon R\rightarrow R'$ and $\ker\varphi$ is generated by a regular sequence. A base case is provided when $\ker\varphi$ is generated by a single $R$-regular element $x\in \m$: in this circumstance, for an $R$-module $M$ which is annihilated by $x$, the construction of Eisenbud--Shamash \cite{Ei80} provides a way to obtain the resolution of $M$ over $R'$ from its resolution over $R$. This construction was used in \cite{BDS23} to prove the periodicity of ideals of minors for modules over complete intersections. With the help of a converse of the Eisenbud--Shamash construction due to Bergh, Jorgensen, and Moore in \cite{BerghJorgensenMoore2020},  an $R$-resolution of an $R'$-module $M$ can be obtained from its $R'$-resolution. In the case $x \in \m \setminus \m^2$, the minimality of the resolution also gets transferred. Using this, we obtain the following result: 
\begin{customtheorem}{C}\label{introthm:deformation}
    Let $(R, \m_R)$ be a  noetherian local ring, $r$ be a positive integer, $x\in\m_R\setminus\m_R^2$ be super-regular, and $M$ be a finitely generated module over $R'=R/(x)$.  Suppose for each $1\leq s \leq r$, there exists $\ell_s\in \mathbb N$ such that $I_{n,s}^{R'}(M)=\m_{R'}^s$ for all $n\geq \ell_s$. 
   If there exists $N\in \mathbb N$ such that $\beta_n^{R'}(M)\geq r$ for all $n \geq N$, then $I_{n,r}^R(M)=\m_R^r$ for all $n\geq \max\{\ell_1, \ldots, \ell_r, N\}$.
\end{customtheorem}

\Cref{introthm:periodicity} and \Cref{introthm:deformation} allow us to extend the periodicity of ideals of minors to certain modules over certain quasi-fiber products and stretched Gorenstein rings of arbitrary dimension (see \Cref{cor:QFP} and \Cref{cor:SGR}).

This article is organized as follows. In \Cref{sec:basic-properties}, we record some basic properties of ideals of minors that are useful for the rest of the article. In Sections \ref{sec:Fiber-Products} and \ref{sec:Stretched-Artin}, we prove the periodicity of ideals of minors over fiber products and artinian stretched Gorenstein rings, respectively. Finally, \Cref{sec:Deformation} focuses on the behavior of periodicity of ideals of minors under deformations.

\section*{Acknowledgements}

We would like to thank Sarasij Maitra and Prashanth Sridhar for helpful discussions regarding this topic. Trung Chau would like to thank his advisor Srikanth Iyengar for his constant encouragement during his Ph.D., as well as the inspiring conversations regarding infinite free resolutions. Trung Chau and Omkar Javadekar acknowledge support from the Infosys Foundation. Souvik Dey was partially supported by the
Charles University Research Center program No.UNCE/SCI/022 and a grant GACR 23-05148S from
the Czech Science Foundation. Ganapathy Krishnamoorthy acknowledges the support from the Prime Minister's Research Fellowship (PMRF) scheme for carrying out this research work. All of the work done in this article took place when Souvik Dey was a research scientist at the Department of Algebra of Charles University, Prague, and he is very grateful for the outstanding atmosphere fostered by the department.

\section{Computational Properties of Ideals of Minors}\label{sec:basic-properties}
\begin{definition}
    Let $R$ be a commutative ring and $A$ be a matrix with entries in $R$.
    The \textit{ideal of $r\times r$ minors} of $A$, denoted $I_r^R(A)$ is the ideal generated by all $r\times r$ matrix minors of $A$.
    By convention, $I_r^R(A)$ is zero when $r$ exceeds the number of rows or columns of $A$, and $I_0^R(A)=R$.\\
    Now suppose that $R$ is a noetherian local ring (resp.~a standard graded algebra over a field), so that every finitely generated (resp.~graded) $R$-module $M$ has a well-defined minimal free resolution.
    Let $\del_n^M$ be the $n^{th}$ differential appearing in the minimal free resolution of $M$, represented as a matrix with respect to some bases.
    Then $I_{n,r}^R(M)$ is the ideal generated by all $r\times r$ minors of $\del_n^M$.
  %  \michael{Jargon/Terminology. Fitting Ideal? This dfn should probably be in the intro?}
\end{definition}

Throughout this section, we let $R$ denote a noetherian local ring (resp.~a standard graded $\mathsf k$-algebra) with the maximal ideal (resp.~the unique homogeneous maximal ideal) $\m$ and the residue field $\mathsf k$. 
We will make use of the following computational properties of ideals of minors.
The proofs of these results involve elementary properties of minors, but we record them below for completeness. 

\begin{proposition}\label{p:MinorsMaxPowers}
    $I_{n,r}^R(M)\subseteq \m^r$.
\end{proposition}

\begin{proof}
    The formula for the determinant of an $r\times r$ matrix is a sum of products of $r$ terms.
    Taking a minimal resolution for $M$, each term of each product can be taken to be in $\m$.
\end{proof}

\begin{lemma}\label{l:SubMatrixMinors}
    Let $A$ and $B$ be matrices with entries in $R$ such that, up to permutation of the rows and columns of $B$,  $A\otimes\id_{\ell\times \ell}$ is a submatrix of $B$.
    Then for every composition $r_1+\dots+r_\ell=r$ of a positive integer $r$, we have $I^R_{r_1}(A)I^R_{r_2}(A)\cdots I^R_{r_\ell}(A)\subseteq I^R_r(B)$.
\end{lemma}
\begin{proof}
    Since ideals of minors are invariant under permutations of rows and columns, we assume that $A\otimes\id_{\ell\times \ell}$ is the $s\times t$ submatrix of $B$ formed by the first $s$ rows and $t$ columns of $B$:

    \[B=
    \left[\begin{array}{ c | c }
    A\otimes\id_{\ell\times \ell} & X \\
    \hline
    Y & Z
  \end{array}\right].
  \]
  The $r\times r$ minors of $B$ whose row indices are less than or equal to $s$ and column indices are less than or equal to $t$ are precisely the $r\times r$ minors of $A\otimes\id_{\ell\times \ell}$, and so $I^R_{r}(A\otimes \id_{\ell\times \ell})\subseteq I^R_r(B)$.
  It therefore suffices to show that $I^R_{r_1}(A)\cdots I^R_{r_\ell}(A)\subseteq I^R_r(A\otimes \id_{\ell\times \ell})$.
  Suppose $A$ is a $p\times q$ matrix.
  If $r_i>p$ or $r_i>q$, then $I^R_{r_i}(A)=0$, and the required inclusion is trivial.
  Therefore, we may assume that each $r_i\leq \min\{p,q\}$.
  Up to a permutation of rows and columns, $A\otimes\id_{\ell\times \ell}$ has the block-diagonal form
  $$A \otimes \id_{\ell \times \ell} =
  \begin{pmatrix}
  A    & 0      & \cdots & 0 \\
  0    & A      & \ddots & \vdots \\
  \vdots & \ddots & \ddots & 0 \\
  0    & \cdots & 0      & A
  \end{pmatrix},$$
  where the matrix $A$ appears $\ell$ times on the diagonal.
  To show $I^R_{r_1}(A)\cdots I^R_{r_\ell}(A)\subseteq I^R_r(A\otimes \id_{\ell\times \ell})$, it is enough to show that if $M_1, \ldots, M_{\ell}$ are minors of $A$ with each $M_i$ of size $r_i$, then the product $M_1 M_2 \cdots M_{\ell} \in I^R_r(A\otimes \id_{\ell\times \ell})$.
  Let $M_i$ be the minor of $A$ formed by rows $\{\alpha_{i1}, \ldots, \alpha_{ir_i}\}$ and columns $\{\gamma_{i1}, \ldots, \gamma_{ir_i}\}$ of $A$.
  Then the product $M_1 M_2 \cdots M_{\ell}$ is the minor of $A \otimes \id_{\ell \times \ell}$ formed by the $r$ rows $\bigcup\limits_{i=1}^{\ell}\{ (i-1)p+\alpha_{ij} \mid 1 \leq j \leq r_i\}$ and the $r$ columns $\bigcup\limits_{i=1}^{\ell}\{ (i-1)q+\gamma_{ij} \mid 1 \leq j \leq r_i\}$.
  Hence $M_1 M_2 \cdots M_{\ell}$ is an $r\times r$ minor of $A\otimes\id_{\ell\times \ell}$, so it lies in $I^R_r(A\otimes \id_{\ell\times \ell})$, as required.
\end{proof}

\begin{proposition}\label{p:SyzSumMinors}
        If $N$ is a direct summand of $\syz^R_{n}(M)$, then for each $m\geq 1$, $I_{m,r}^R(N)\subseteq I_{{n+m},r}^R(M)$. 
\end{proposition}

\begin{proof}
    We may assume that $r$ and $m$ are chosen so that $I_{m,r}^R(N)$ is non-zero as otherwise the statement is trivial.

    The $n^{th}$ shift of a resolution of $N$ is a direct summand of the resolution of $M$, so $\syz^R_{m-1}(N)$ is a direct summand of $\syz^R_{n+m-1}(M)$.
    So it suffices to consider the case $m=1$.

    Since $N$ is a direct summand of $\syz^R_{n}(M)$, if  $\del_{n+1}^M$ and $\del_{1}^N$ are the matrices in minimal resolutions of $M$ and $N$, respectively, then $\del_{n+1}^M$ can be written in block form as 
    \[\del_{n+1}^M=\begin{bmatrix}
        \del_1^N & 0 \\
        0        & A
    \end{bmatrix}
    \]
    for some matrix $A$. Since $\del_{1}^N$ is a submatrix of $\del_{n+1}^M$, the result follows.
    %{The result now follows from \Cref{l:SubMatrixMinors}.}
\end{proof}

 We conclude this section with a relation between trace ideal of modules and ideal of minors. Recall that for an $R$-module $M$, the trace ideal of $M$ is $\tr_R(M):=\sum_{f\in \Hom_R(M,R)} f(M)$, see \cite[Section 2]{lindo} for instance.

\begin{proposition}\label{trace}
    Let $M$ be an $R$-module. Then, $I^R_{n,1}(M)\subseteq \tr_R(\syz_{n}^R (M))$ for all $n\ge 1$. If $(R,\m)$ is local, $\projdim_R(M)=\infty$,  and $I^R_{n,1}(M)=\m$ for some $n\ge \depth (R)+1$, then $\tr_R(\syz_n^R (M))=\m$.   
    
\end{proposition}

\begin{proof} Since $I^R_{n,1}(M)=I^R_{1,1}(\syz_{n-1}^R(M))$ for all $n\ge 1$,  it is enough to prove $I^R_{1,1}(M)\subseteq \tr_R(\syz^R (M))$.  Let $G\xrightarrow{\del} F\to M\to 0$ be a minimal free presentation of $M$. Choose standard bases of $F$ and $G$, and let $\{\pi_i\}$ be the collection of all coordinate projections $\pi_i:F\to R$. Then, $$I^R_{1,1}(M)=\sum_{\pi_i}\im(\pi_i\circ \del)=\sum_{\pi_i}\pi_i|_{\im(\del)}\left(\im(\del)\right)\subseteq \tr_R(\im(\del))=\tr_R(\syz^R (M))$$ where the penultimate inclusion follows since $\pi_i|_{\im(\del)}\in \Hom_R(\im(\del),R)$. Now to prove the last part of the claim, it follows from the inclusion that $\tr_R(\syz_n^R (M))=\m$, or $R$. In view of \cite[Proposition 2.8(iii)]{lindo}, it is then enough to show that $\syz_n^R (M)$ does not have $R$ as a summand for $n\ge \depth (R)+1$. Indeed, if  $\syz^n_R (M)$ did have $R$ as a summand, then $(\syz_n^R (M))/(\mathbf x)(\syz_n^R (M))$ would be faithful over $R/(\mathbf x)$ for every $R$-regular sequence which would contradict \cite[Lemma 5.8, Definition 5.2]{dey2024finitehomologicaldimensionhom} since $n\ge \depth (R)+1$.  
\end{proof}

\section{Periodicity of Ideals of Minors over Fiber Products}\label{sec:Fiber-Products}

\begin{definition}
Let $(S,\p,\mathsf k)$ and $(T,\q,\mathsf k)$ be local rings with a common residue field. %or standard graded $\mathsf k$-algebras.
The \emph{fiber product} of $S$ and $T$ over $\mathsf k$, denoted $S\times_{\mathsf k} T$, is the pullback of the natural projections $S\to \mathsf k$ and $T \to \mathsf k$. It comes equipped with projections to $S$ and $T$, as in the following diagram:
\[
\begin{tikzcd}
  S \times_{\mathsf k} T \arrow[r, "\pi_S"] \arrow[d, "\pi_T"']
    & S \arrow[d] \\
    T \arrow[r]
    & \mathsf k
\end{tikzcd}.
\]
\iffalse
% Not sure what direction the pushout/pullback arrow should be!!
\[
\begin{tikzcd}
  S \times_{\mathsf k} T \arrow[r] \arrow[d]
    \arrow[dr, phantom, "\ulcorner", very near start]
    & S \arrow[d] \\
    T \arrow[r]
    & \mathsf k
\end{tikzcd}
\]
\fi
Concretely, $S\times_{\mathsf k} T$ is the subring of the ordinary cartesian product $S\times T$ consisting of the set of all pairs $(s,t)$ such that $\bar{s}=\bar{t}$ in $\mathsf k$.
\end{definition}

Throughout the remainder of this section, $R=S\times_{\mathsf k} T$ will be the fiber product of noetherian local rings $S$ and $T$. %We let $e, e_1$ and $e_2$ denote $\embdim(R), \embdim(S)$ and $\embdim(T)$, respectively. 
We also assume that all fiber products are non-trivial, i.e., with $S \not\cong \mathsf k$ and $T \not \cong \mathsf k$. Thus, we have $\embdim(S)\geq 1$ and $\embdim(T) \geq 1$. In particular this forces $\embdim(R)\geq 2$.

While we state and prove the results for fiber products of local rings, they hold true even over fiber products of standard graded algebras over a field.
\begin{remark}
    The fiber product $R$ is also a local ring with residue field $\mathsf k$ and maximal ideal $\p\oplus \q$.
    In particular, there are injective maps $\iota_\p$ from $\p$ and $\iota_\q$ from $\q$ into $R$.
\end{remark}
\iffalse
We introduce notation for certain lifts of maps from the factors to the fiber product.

\begin{notation}
    Let $\del'\colon F_1'\rightarrow F_2'$ be a minimal map of free $S$-modules of rank $\beta_1$ and $\beta_2$, respectively, and let $F_1$ and 
    Then $\del=\iota_\p^{\oplus\rank F_2}\del'\pi_S^{\oplus\rank F_1}$ is a map from free $R$-modules $\varphi\colon F_1\rightarrow F_2$ with the same ranks as $F_1$ and $F_2$.
    The same construction works for a minimal map of free $T$-modules.
    Note that minimality is essential, as there are no ring maps from the factors $S$ and $T$ back into the fiber product $S\times_kT$.
    Diagramatically, $\del$ is defined 

    \[
\begin{tikzcd}
    F'_1\arrow[r,"\del'"] & F'_2\arrow[d,"\iota_S^{\oplus\rank F_2}"]\\
    F_1\arrow[r,"\del", dotted] \arrow[u,"\pi_S^{\oplus \rank F_1}"]& F_2
\end{tikzcd}
    \]
    Given $F_2\xrightarrow{\del_1vz}$
\end{notation}
\fi

\begin{remark}\label{r:Lift}
A minimal map $\del\colon S^{m_1}\rightarrow S^{m_2}$ of free $S$-modules (or $T$-modules) can be canonically lifted to a map $\widetilde{\del}\colon R^{m_1}\rightarrow R^{m_2}$.
Representing $\del$ as a matrix, each entry $\del_{i,j}$ is in $\p$ and hence each pair $(\del_{i,j}, 0)$ is in $S\times_{\mathsf k} T$.
Abusing notation, we write $\del_{i,j}$ for $(\del_{i,j},0)$ and will use the same matrix to represent $\widetilde{\del}$.
This can be extended to lifting a minimal complex of $S$-modules to a minimal complex of $R$-modules.
\end{remark}

\begin{example}\label{e:LiftedComplex}
    Let $S=\mathsf k[[x,y]]$, $T=\mathsf k[[z]]$, and $R=S\times_{\mathsf k} T \cong \mathsf k[[x,y,z]]/(xz,yz)$.
    The minimal resolution of $\mathsf k$ over $S$ is $$0\to S\xrightarrow{\begin{bmatrix}-y\\x\end{bmatrix}}S^2\xrightarrow{\begin{bmatrix}x & y\end{bmatrix}}S \to 0.$$
    The lifted complex is $$ 0 \to R\xrightarrow{\begin{bmatrix}-y\\x\end{bmatrix}} R^2\xrightarrow{\begin{bmatrix}x & y\end{bmatrix}}R\to 0.$$
\end{example}

The lifted complex of \Cref{e:LiftedComplex} is no longer acyclic.
The construction of a resolution over $R$ also requires the resolutions of $\mathsf k$ over $S$ and $T$, as described by Moore \cite[Construction~1.7]{moore}.
Our later arguments will rely on the structure of the matrices arising from this resolution, so we provide a self-contained treatment of this resolution in notation better suited to our subsequent analysis.

\begin{construction}[\cite{moore}]\label{c:FiberProductRes}
Let $R=S\times_{\mathsf k} T$ and $M$ be an $S$-module.
Let $E_\bullet$ minimally resolve $\sk$ over $S$, $F_\bullet$ minimally resolve $\sk$ over $T$, and $P_\bullet$ minimally resolve $M$ over $S$.
As described in \Cref{r:Lift}, lift these to complexes of free $R$-modules.
%\ganapathy{The resolution $G_\bullet$ of $M$ over $R$ is formed using tensors of the above complexes, but the differential is not the ordinary differential of the tensor of complexes.}
The underlying module structure of the minimal resolution $G_\bullet$ of $M$ over $R$ is then given by 
\[ F_\bullet\otimes_R \mathbb{T}_R\left(E_{\geq 1}\otimes_R F_{\geq 1}\right)\otimes_R P_\bullet \]
where $\mathbb{T}_R\left(E_{\geq 1}\otimes_R F_{\geq 1}\right)$ is the tensor algebra over $R$ with $E_{\geq 1}\otimes_R F_{\geq 1}$ viewed as a free $R$-module graded by homological degrees.
Note that $F_0=R$ and $\mathbb{T}^0_R\left(E_{\geq 1}\otimes_R F_{\geq 1}\right)=R$.
Identifying $r\otimes f$ with $f$ when $r$ is an element of $R$, we may think of the simple tensors of $G_\bullet$ as  starting with any element of $F_\bullet$, $E_{\geq 1}$, or $P_\bullet$; the differential of $G_\bullet$ is given by $\del^F$, $\del^E$, or $\del^P$, respectively, acting on the leftmost tensor factor. In block matrix form, $\del_n^G$ is given by
\[
\begin{bmatrix}
  \del_n^P & \del_1^F\otimes P_{n-1}   &           0                         &             0            & \dots & 0 & 0 \\
  0        &             0             & \del_1^E \otimes F_1\otimes P_{n-2} & \del_2^F \otimes P_{n-2} & \dots & 0 & 0 \\
  \vdots   &          \vdots           &               \vdots                &        \vdots            & \dots & \vdots  & \vdots  \\
  0        &             0             &               0                     & 0  & \dots & \del_1^E\otimes F_{n-1}\otimes P_0 & \del^F_n\otimes P_0 
\end{bmatrix}
\]

In the above matrix, all tensors are over $R$.
Each row has two non-zero blocks and each column has one non-zero block. The tensor of the differentials $\del^F$ and $\del^E$ with free modules means that each block is itself block diagonal, with copies of the matrices for $\del^F$ or $\del^E$ along the diagonal, with the number of copies equal to the rank of the free module.
\end{construction}

% Have examples of matrices for particular homological degrees before having the big block matrix

We now prove the main theorem of this section.

\begin{theorem}\label{t:FiberProduct}
    Let $R=S \times_{\mathsf k} T$ with $e=\embdim(R)\geq 3$, $\mathfrak m$ the maximal ideal of $R$, and $M$ be a finitely generated $R$-module of infinite projective dimension. Then $I_{n,r}^R(M)=\m^r$ for all $n\geq \dfrac{2r}{\embdim(S) \embdim(T)}+8$.
\end{theorem}

\begin{proof}
The inclusion $I_{n,r}^R(M)\subseteq \m^r$ is given by \Cref{p:MinorsMaxPowers}, so it suffices to show that $I_{n,r}^R(M)\supseteq \m^r$.

With $\edim(S)=e_1$ and $\edim(T)=e_2$, since $e=e_1+e_2\geq 3$ we may assume that $e_1\geq 2$.
From \cite{dress/kramer} (see also \cite[4.2]{moore}) we know that $\syz_2^R(M)$ decomposes as the direct sum of an $S$-module and a $T$-module. In the event that the $S$-module in this decomposition is zero\footnote{For instance, $R=\mathsf k[[w,x,y,z]]/(wy,wz,xy,xz,z^2)$, $S=\mathsf k[[w,x]]$, $T=\mathsf k[[y,z]]/(z^2)$, and $M=R/(x+z)$ satisfies $\syz_2^R(M)\cong \mathsf k[[y]]$.}, the next syzygy decomposes as $\syz_3^R(M)\cong\syz_1^T(\syz_2^R(M))\oplus S^{\oplus \beta_2^R(M)}$.
We see in either case that it is possible to obtain a non-zero $S$-module $N$ which is a direct summand of either $\syz_2^R(M)$ or $\syz_3^R(M)$.
Using \Cref{p:SyzSumMinors}, it suffices to show that $I_{n,r}^R(N)\supseteq\m^r$ for $n\geq \frac{2r}{e_1e_2}+5$.

We use the block decomposition of the differential $\del^G_n$ of the resolution $G_\bullet$ of $N$ built in \Cref{c:FiberProductRes}. When $n$ is even, we focus on the two blocks
\[\del_1^F\otimes (E_1\otimes F_1)^{\otimes \frac{n-4}{2}}\otimes E_2\otimes F_1 \otimes P_0
\quad\quad\quad\text{and}\quad\quad\quad
\del_1^E\otimes(F_1\otimes E_1)^{\otimes \frac{n-2}{2}}\otimes F_1\otimes P_0\]
and when $n$ is odd, we focus on the two blocks %\michael{Need to handle possibility that $P_2=0$.  Either work with a summand of a syz of $M$ so that $P_2$ is non-zero, or switch to something else}
\[\del_1^E\otimes (F_1\otimes E_1)^{\otimes \frac{n-5}{2}}\otimes F_1 \otimes E_2 \otimes F_1 \otimes P_0
\quad\quad\quad\text{and}\quad\quad\quad
\del_1^F\otimes(E_1\otimes F_1)^{\otimes \frac{n-1}{2}}\otimes P_0.\]

Substituting in the bound on $n$, there are at least $\lceil \frac{r}{e_1e_2} \rceil$-many tensor factors of $E_1 \otimes F_1$ (respectively, $F_1\otimes E_1$) in the blocks above (in both even and odd cases).
Each of these factors has rank $e_1e_2$, so each of these blocks is of the form $\del_1^F \otimes R^\ell=\del_1^F\otimes\id_{\ell\times\ell}$ and $\del_1^E \otimes R^\ell=\del_1^E\otimes\id_{\ell\times\ell}$, respectively, where $\ell\geq r$.
In either case, we are done by \Cref{l:SubMatrixMinors} (where we take each part $r_i$ of the composition to be one or zero, and recall that $I^R_0(A)=R$ for any matrix $A$).
\end{proof}

Observe that for a given fiber product $R$, the bound for $n$ obtained in the theorem above is independent of the choice of the module $M$.  

\begin{remark}\label{rem:betti-nos-over-FP}
    With the same hypotheses as \Cref{t:FiberProduct}, from the proof of \Cref{t:FiberProduct}, we see that given any $R$-module $M$ of infinite projective dimension, and $r \in \mathbb N$, we have $\beta_n^R(M) \geq r$ for $n\gg 0$. 
\end{remark}

%We remind the reader that for $\edim(R)\le 2$, the behavior of ideal of minors is two-periodic by \cite{BDS23} as mentioned in the Introduction. 
The next example shows that the hypothesis $\embdim(R)\geq 3$ in \Cref{t:FiberProduct} cannot be removed. 

\begin{example}
    Consider the fiber product ring $R=\sk[[x]] \times_{\mathsf k} \sk[[y]] \cong \dfrac{\sk[[x,y]]}{(xy)}$. Then for the $R$-module $M=R/(x)$, we see that $$I^R_{n,1}(M)= \begin{cases}
         (x) & \text{\ if\ } n \text{\ is\ odd} \\
         (y) & \text{\ if\ } n \text{\ is\ even.}
    \end{cases}
   $$ 
\end{example}

\begin{corollary}
    Let $R$ be a fiber product. Then given any finitely generated $R$-module $M$ with infinite projective dimension and $r \in \mathbb N$, we have $I^R_{n,r}(M)=I^R_{n+2,r}(M)$ for $n\gg 0$.
\end{corollary}
\begin{proof}
    When $\edim(R)\geq 3$, \Cref{t:FiberProduct} shows that $I^R_{n,r}(M)=\m^r$ for $n\gg 0$. When $\edim(R)\leq 2$, $R$ must be a complete intersection or a Golod ring. Then the result follows from \cite[Theorem 1.2 and 1.3]{BDS23}.
\end{proof}

%\begin{remark}\label{r:Edim2Case}
%Let $R$ be a local ring with $\edim(R)-\depth(R)\le 2$. Then, in view of   \cite[5.1, Proposition 5.3.4]{Avramov1998}, $R$ is either a complete intersection or a Golod ring. 
%Thus, when $\embdim(R)-\depth(R)\le 2$, given any finitely generated $R$-module $M$, we have $I_{n,r}^R(M) = I_{n+2, r}^R(M)$ for $n \gg 0$ by \cite{BDS23}.  
%\end{remark}

\section{Periodicity of Ideals of Minors over Artinian Stretched Gorenstein Rings}\label{sec:Stretched-Artin}

\begin{definition}[\protect{\cite{Sa79}}]
Let $(R,\m,\mathsf k)$ be a noetherian Cohen--Macaulay local ring of dimension $d$.

\begin{enumerate}[(a)]
    \item When $d=0$, we say that $R$ is \emph{stretched} if $\lambda(R)- \embdim (R)$ is the least integer $i$ such that $\m^{i+1}=0$, where $\lambda(R)$ denotes the length of $R$.
    \item When $d>0$, we say that $R$ is \emph{stretched} if there is a minimal reduction $\underline{x}=x_1, \ldots, x_d$ of $\m$ such that $R/(\underline{x})$ is stretched. 
\end{enumerate}  
\end{definition}
Note that regular local rings are not stretched. Also, if $\dim(R)=0$, then $R$ is stretched if and only if $\m^2$ is principal.

In this section, we focus on the periodicity of ideals of minors for modules over artinian stretched Gorenstein rings. The case $\dim(R)>0$ is addressed in \Cref{sec:Deformation}.

 The following theorem provides the structure of stretched Gorenstein rings (see \cite{Sa79}).

\begin{theorem}\label{thm:structure-of-stretched-Gorenstein-rings}
    Let $(S,\mathfrak n)$ be a regular local ring, and $I \subseteq \mathfrak n^2$ be such that $R=S/I$ is an artinian stretched Gorenstein ring. Let $\m=\mathfrak n/I$ and $e=\edim(S)=\edim(R)\geq 2$. If $\ch(S/\mathfrak n) \neq 2$, then there exists a minimal generating set $\{x_1,\ldots, x_e\}$ of $\mathfrak{n}$ such that
        $$R \cong S/ (\{ x_ix_j \mid i \neq j\}, \{x_1^s-u_ix_i^2 \mid i > 1\}),$$ where the elements $u_i$ are units in $S$, and $\m^s=\soc(R)$.  
\end{theorem}
\begin{remark}\label{rem:stretched-Gorenstein-remark} Let $(R,\m,\sk)$ be an artinian stretched Gorenstein ring with $\edim(R)=e\geq 2$ and $\ch(\sk) \neq 2$. %\hfill{}
%\begin{enumerate}[(a)]
     Note that by the Cohen structure theorem, artinian local rings are quotients of regular local rings. Thus, \Cref{thm:structure-of-stretched-Gorenstein-rings} includes the class of all artinian stretched Gorenstein rings. For the remainder of this section, we will adopt the notation of \Cref{thm:structure-of-stretched-Gorenstein-rings}.
    %\item There exists a minimal generating set $\{ x_1, x_2, \ldots, x_e\}$ of $\m$ with $x_ex_j=0$ for $1\leq j\leq e-1$ and $x_e^3=0$.\michael{This holds not just for $x_e$, but for every $x_i$ with $i>1$.  Also, the way you wrote it above doesn't make it seem like it depends on generating set, but you say that it does here.  I would move the "some generating set" comment to inside Theorem 5.2.} Furthermore, 
    Observe that in this case, for all $i\geq 2$, the ideal $\m^i$ is generated by $x_1^i$.
%\end{enumerate} 
\end{remark}

Note that  artinian Gorenstein rings of embedding dimension at most $2$ are complete intersection rings (see \cite[Corollary 21.20]{Eisbook} or \cite[Proposition 5.3.4]{Avramov1998}). On the other hand, if $R$ is an artinian stretched Gorenstein ring and $\edim(R) \geq 3$, then from \Cref{thm:structure-of-stretched-Gorenstein-rings}, we see that $R$ is not a complete intersection ring.

We recall an important ingredient in our proofs in the sequel.

\begin{proposition}[\protect{\cite[Proposition 2.2]{GP90}}]\label{prop:strict-inequality-of-Betti-numbers} 
 Let $(R,\m, \mathsf k)$ be an artinian local ring with $\edim(R) = e$ and length $l$. Set $h = \min\{ i \mid \m^{i+1} = 0 \}$. If $M$ is a finitely generated $R$-module, then
 $$\beta_{n+1}(M) \geq (2e-l+h-1) \beta_n(M)$$
 for $n \geq \mu(M)$.
\end{proposition}

\begin{remark}\label{rem:betti-no-over-artinian-stretched}
Let $(R, \m,\mathsf k)$ be an artinian stretched Gorenstein ring with $\ch(\mathsf k)\neq 2$, $\embdim(R)=e\geq 3$, length $l$, and set $h=\min\{i \mid \m^{i+1} =0\}$. Then by \Cref{prop:strict-inequality-of-Betti-numbers}, given any finitely generated $R$-module $M$, for all $n \geq \mu(M)$ we have $$\beta_{n+1}(M) \geq (2e-l+h-1)\beta_n(M) = (e-1) \beta_n(M) \geq 2 \beta_n(M).$$ 
\end{remark}

In particular, if $M$ is not free, the sequence of Betti numbers of $M$ is eventually  strictly increasing. In our case, this phenomenon guarantees the syzygy modules of $M$ to contain generators with special properties, a fact that will be used in the sequel.

\begin{lemma}\label{lemma:existence-of-annihilator-variable}
Let $(R,\m, \sk)$ be a noetherian local ring and $\varphi\colon R^n\to R^m$ be such that $\varphi(R^n)\subseteq \m R^m$. 
Suppose that \begin{enumerate}[{\rm (a)}]
    \item $\mu( \varphi(R^n))=n$;
    \item $m<n$;
    \item there exists a minimal generating set $\{ x_1, x_2, \ldots, x_e\}$ of $\m$ with $x_ex_j=0$ for $1\leq j\leq e-1$ and $x_e^3=0$.
\end{enumerate}  Then $x_e u=0$ for some minimal generator $u$ of $\varphi(R^n)$.
\end{lemma}
\begin{proof}
   Let $A$ be the matrix representation of $\varphi$ with respect to the standard basis of $R^n$ and $R^m$. We now show that by a suitable change of basis of $R^n$, the matrix $A$ can be reduced to a matrix $A'$ such that every entry in the last $(n-m)$ columns of $A'$ belongs to the ideal $(x_1, x_2, \ldots, x_{e-1},x_e^2)$.
    Note that every invertible column operation on $A$ corresponds to a change of basis of $R^n$. Here, we obtain a desired basis of $R^n$ by performing a series of column operations of the following two types: (i) swapping two columns, (ii) adding an $R$-multiple of one column to another column.
    
    Observe that we have $\dim_{R/\mathfrak m}(\m R^m / (x_1,x_2,\ldots,x_{e-1},x_e^2) R^m)= m$.
    Let ${\bar A}$ denote the matrix over $R/\m$ obtained by taking the images of the entries of $A$ in $\m R^m / (x_1,x_2,\ldots,x_{e-1},x_e^2) R^m$. 
    Let the column operations $\sigma_1,\sigma_2,\dots, \sigma_t$ reduce ${\bar A}$ to a matrix whose last $(n-m)$ columns are zero. Then, the lifts of the operations $\sigma_1,\sigma_2,\dots, \sigma_t$ performed on $A$ reduce it to a matrix $A'$ whose last $(n-m)$ columns are in $(x_1, x_2, \ldots, x_{e-1},x_e^2)R^m$. Since $x_ex_j=0$ for $1\leq j\leq e-1$ and $x_e^3=0$, the element $x_e$ annihilates $(x_1, x_2, \ldots, x_{e-1},x_e^2)R^m$. Thus, letting $u$ be the last column of $A'$, we see that $x_eu = 0$.  Note that due to (a), $u$ is a minimal generator of $\varphi(R^n)$, as desired.
    \end{proof}

\begin{proposition}\label{prop:existence-of-annihilator-variable-for-stretched-Gorenstein}
    Let $(R,\m,\sk)$ be an artinian stretched Gorenstein ring with $\edim(R)=e\geq 3$, $\ch(\sk) \neq 2$, and $M$ be a finitely generated non-free $R$-module. Then, with the notation as in \Cref{thm:structure-of-stretched-Gorenstein-rings}, given any $n \geq \mu(M)$, there is a minimal generator $u$ of $\Omega_{n+1}(M)$ such that $x_eu = 0$.
\end{proposition}
\begin{proof}
 %   Since $R$ is stretched Gorenstein, we have $2e-l+h-1=e-1>1$. Thus, by \Cref{prop:strict-inequality-of-Betti-numbers}, for every $n \geq \mu(M)$, 
 By \Cref{rem:betti-no-over-artinian-stretched}, for every $n \geq \mu(M)$ we have $\beta_{n+1}(M) > \beta_n(M)$. Let $(F_n, \del_n)$ be a minimal free resolution of $M$. Then, applying \Cref{lemma:existence-of-annihilator-variable} to $\varphi= \del_{n+1}\colon F_{n+1}\to F_n$,  we get the required result.
\end{proof}

\begin{remark}\label{remark:minimality-of-linear-generators}
    Let $F$ be a finitely generated free $R$-module mapping minimally onto an $R$-module $M$ via a map $\varphi$. If $z \in \ker(\varphi) \cap (\m F \setminus \m^2 F)$, then $z$ is a minimal generator of $\ker(\varphi)$. This is clear since $\ker(\varphi)\subseteq \m F$ by the minimality of $\varphi$. 
    
    Furthermore, the set $\{z_1,\ldots, z_t\}\subseteq \ker(\varphi) \cap (\m F \setminus \m^2 F)$ forms a part of a minimal generating set of $\ker(\varphi)$ if their images in $\m F / \m^2F$ form an $R/\m$-linearly independent set. To see this, 
    let $a_1,\ldots,a_t \in R$ be such that $\sum\limits_{j=1}^t a_j z_j\in \m \ker(\varphi)$. Since $\ker(\varphi)\subseteq \m F$, we see that $\sum\limits_{j=1}^t a_j z_j \in \m^2 F$. 
    Now, the $R/\m$-linear independence of the images of $z_1, \ldots, z_t$ gives $a_{1}, \ldots, a_{t} \in \m$. This proves that $\{z_1, \ldots, z_t\}$ forms a part of a minimal generating set of $\ker(\varphi)$.
\end{remark}

Next we prove the periodicity of ideals of minors in the case $M$ is a module with a special generator.

\begin{lemma}\label{lemma:periodicity-for-x_e}
Let $(R, \m, \sk)$ be an artinian stretched Gorenstein local ring, $e=\edim(R)\geq 3$, $\ch(\sk) \neq 2$, and $t\in \mathbb N$.  If $N\subseteq \m R^t$ is  such that $( x_e, 0, \ldots,0)\in R^t$ is a minimal generator of $N$, then for every $r\in \mathbb N$, we have $I^R_{n,r}(N)=\m^r$ for $n\gg 0$. 
\end{lemma}
\begin{proof}
We inductively build a minimal free resolution $F_\bullet$ of $N$ with basis $\{\omega^{(n)}_j \mid 1\leq j \leq \beta_n(N)\}$ and with differentials $\del_n\colon F_n\rightarrow F_{n-1}$, as follows.
Consider the minimal onto map $\del_0\colon F_0\to N$ such that $\del_0(\omega_1^{(0)})=(x_e, 0,\ldots, 0)$.  
By \Cref{remark:minimality-of-linear-generators}, the set $\{x_1\omega_1^{(0)},x_2\omega_1^{(0)}\}$ %\mid 1\leq j \leq e-1\}$
is a part of a minimal generating set for $\Omega_1(N)$.  
In a similar way, letting $\del_1 \colon F_1\to \Omega_1(N)$ be a minimal onto map with $\del_1(\omega_1^{(1)})=x_2\omega_1^{(0)}$ and $\del_1(\omega_2^{(1)})=x_1\omega_1^{(0)}$,  by \Cref{remark:minimality-of-linear-generators}, we get that $\{x_1\omega_{1}^{(1)}, x_2\omega_{2}^{(1)}, x_3\omega_{2}^{(1)}\}$ is a part of a minimal generating set of $\Omega_2(N)$. Before continuing, we have the following claim.

\begin{claim}\label{clm:syzygy}
    If $\{ x_1 \omega_{j}^{(n)}, x_2 \omega_{\gamma + k}^{(n)}, x_3 \omega_{\gamma + k}^{(n)} \mid 1 \leq j \leq \gamma, 1 \leq k \leq \delta \}$ is a part of a minimal generating set of $\Omega_{n+1}(N)$, then there exists a minimal onto map $\del_{n+1}\colon F_{n+1}\to \Omega_{n+1}(N)$ such that 
\begin{enumerate}[(i)]
    \item  $\del_{n+1}(\omega_k^{(n+1)})= x_2\omega_{\gamma + k}^{(n)}$ for $1\leq k \leq \delta$,
    \item  $\del_{n+1}(\omega_{\delta+k}^{(n+1)})=  x_3\omega_{\gamma + k}^{(n)}$ for $1\leq k \leq \delta$, 
    \item  $\del_{n+1}(\omega_{2\delta+j}^{(n+1)})= x_1\omega_{j}^{(n)}$ for $1\leq j \leq \gamma$,
\end{enumerate}

and $\{ x_1 \omega_{k}^{(n+1)}, x_2 \omega_{2 \delta + j}^{(n+1)}, x_3 \omega_{2 \delta + j}^{(n+1)} \mid 1 \leq k \leq 2 \delta, 1 \leq j \leq \gamma \}$ is a part of a minimal generating set of $\Omega_{n+2}(N)$. 
\end{claim}

The proof of Claim~\ref{clm:syzygy} is similar to the way we obtained a part of a minimal generating set of $\Omega_2(N)$ above, so we will leave the details to interested readers.
\par Since $\{x_1\omega_{1}^{(1)}, x_2\omega_{2}^{(1)}, x_3\omega_{2}^{(1)}\}$ is a part of a minimal generating set of $\Omega_2(N)$, by applying Claim~\ref{clm:syzygy} repeatedly, we get that for any $n \in \mathbb N$, 
\begin{enumerate}[(a)]
    \item $\{ x_1 \omega^{(2n)}_{j} \mid 1 \leq j \leq 2^n\}$ is a part of a minimal generating set of $\Omega_{2n+1}(N)$.
    \item $\{ x_1 \omega^{(2n+1)}_{j} \mid 1 \leq j \leq 2^n\}$ is a part of a minimal generating set of $\Omega_{2n+2}(N)$.
\end{enumerate}

For $r\geq 2$, let $s\geq \log_2 r$. Then, for all $n \geq 2s+2$, we see that $x_1 {\id_{r \times r}}$ is a submatrix of a minimal presentation matrix of $\Omega_{n-1}(N)$, which shows $(x_1^r) = \m^r \subseteq I^R_{n,r}(N)$. Since $I^R_{n,r}(N) \subseteq \m^r$, we get $I^R_{n,r}(N) = \m^r$.

For $r=1$, observe that for any $n \geq 2$, the set $\{ x_1 \omega_{j}^{(n)}, x_2 \omega_{k}^{(n)} \}$ for some $1 \leq j,k \leq \beta_{n}(N)$ is a part of a minimal generating set of $\Omega_{n+1}(N)$. Then, given any $1 \leq i \leq e$, $x_i$ annihilates one of these minimal generators. Hence, there exists a minimal presentation matrix of $\Omega_{n+1}(N)$ with $x_i$ as one of its entries for every $i$. Therefore, $I^R_{n,1}(N) = \m$ for all $n \geq 3$.
\end{proof}

We are now in a position to prove our main result of this section.

\begin{theorem}\label{thm:main-theorem-stretched-Gorenstein}
    Let $(R,\m, \sk)$ be an artinian stretched Gorenstein ring with $\edim(R)=e\geq 3$ and $\ch(\sk) \neq 2$. Then given any finitely generated non-free $R$-module $M$ and $r \in \mathbb N$, we have $I^R_{n,r}(M)=\m^r$ for $n\gg 0$.
\end{theorem}
\begin{proof}
By \Cref{prop:existence-of-annihilator-variable-for-stretched-Gorenstein}, there is a minimal generator $u$ of $\Omega_{\mu(M)+1}(M)$ such that $x_e u=0$. Let $\{u=u_1, \ldots, u_{\beta_{\mu(M)+1}(M)}\}$  be a minimal generating set of $\Omega_{\mu(M)+1}(M)$. Let $N \coloneqq \Omega_{\mu(M)+2}(M)\subseteq \m R^{\beta_{\mu(M)+1}(M)}$. Then by \Cref{remark:minimality-of-linear-generators}, $(x_e, 0,\ldots,0)\in R^{\beta_{\mu(M)+1}(M)}$ is a minimal generator of $N$. Hence, by   \Cref{lemma:periodicity-for-x_e}, we get that $I^R_{n,r}(N)=\m^r$ for $n \gg 0$. Therefore, the proof is complete since $I^R_{n,r}(N)= I^R_{n+\mu(M)+2,r}(M)$.
\end{proof}

Let $M$ be a finitely generated module over a noetherian local ring $(R,\mathfrak m)$. Our result is similar to \cite[Theorems 1.2 and 1.3]{BDS23}, which state that if $R$ is a complete intersection ring or a Golod ring, we have $I^R_{n,r}(M) = I^R_{n+2,r}(M)$ for all $n \gg 0$. In our case when $R$ is an artinian stretched Gorenstein ring with $\embdim(R)\geq 3$, Theorem~\ref{thm:main-theorem-stretched-Gorenstein} is a stronger result:   $I^R_{n,r}(M)=\mathfrak{m}^r$ for $n\gg 0$. However, when $\embdim(R)\leq 2$, this is no longer true, as shall be seen in the next examples.

\begin{example}\label{example:periodicity-for-edim-one-and-two}
\hfill{}
\begin{enumerate}[(a)]
    \item Consider the  artinian  stretched Gorenstein ring $R=\sk[[x]]/(x^3)$. Then, for the $R$-module $M= \sk$, we have $$I^R_{n,1}(M)= \begin{cases}
         (x) & \text{\ if\ } n \text{\ is\ odd} \\
         (x^2) & \text{\ if\ } n \text{\ is\ even.}
    \end{cases}
   $$  
   \item Consider the  artinian  stretched Gorenstein ring $R=\sk[[x_1,x_2]]/(x_1x_2, x_1^2-x_2^2)$. Then, for the $R$-module $M=R/(x_1)$, we have $$I^R_{n,1}(M)= \begin{cases}
         (x_1) & \text{\ if\ } n \text{\ is\ odd} \\
         (x_2) & \text{\ if\ } n \text{\ is\ even.}
    \end{cases}$$
\end{enumerate}
\end{example}

\begin{corollary}
    Let $(R,\m, \sk)$ be an artinian stretched Gorenstein ring with $\ch(\sk) \neq 2$. Then given any finitely generated non-free $R$-module $M$ and $r \in \mathbb N$, we have $I^R_{n,r}(M)=I^R_{n+2,r}(M)$ for $n\gg 0$.
\end{corollary}
\begin{proof}
    If $\edim(R)\geq 3$, \Cref{thm:main-theorem-stretched-Gorenstein} shows that we have $I^R_{n,r}(M)=\m^r$ for $n\gg 0$, and thus the result follows trivially. Now assume that $\edim(R)\leq 2$. Then $R$ is either a Golod ring or a complete intersection ring. In fact, in this case $R$ must be a complete intersection as Gorenstein + Golod = hypersurface (see, e.g., \cite[page 47]{Avramov1998}). The result then follows from \cite[Theorem 1.2]{BDS23}.
\end{proof}

The stretched assumption cannot be removed, as there exist artinian Gorenstein rings for which even the $k$-periodicity of minors does not hold for any integer $k$. \cite[Proposition 5.1 (2)]{BDS23} provides one such example (also see \cite[Proposition 3.1 (i)]{GP90}).

Let $(R, \m, \sk)$ be an artinian stretched Gorenstein ring. Using the proofs of  \Cref{lemma:periodicity-for-x_e} and \Cref{thm:main-theorem-stretched-Gorenstein}, for a given finitely generated $R$-module $M$ and $r\in \mathbb N$, it is possible to get a lower bound $m$ on the homological degree,  in terms of $r, \edim(R)$, and $\mu(M)$ so that $I^R_{n,r}(M)=\m^r$ for all $n\geq m$. But there does not exist any such uniform bound that is independent of $M$. We provide such an example.

\begin{example}\label{exam:no-uniform-bound-stretched-gorenstein}
    Let $(R, \m, \sk)$ be an artinian stretched Gorenstein ring. For each $n\in \mathbb N$, we will show that there is an $R$-module $M_n$ such that $I^R_{n,1}(M_n)= \soc(R)$. Thus, whenever $\m^2\neq 0$, we get that $I^R_{n,1}(M_n)\neq \m$. 
    
     Let $\soc(R)=(x)$, and $\cdots \rightarrow  F_2 \xrightarrow{\del_2} F_1 \xrightarrow[]{\del_1} R \to 0$ be a minimal projective resolution of $\sk$. Since $R$ is artinian Gorenstein, $(-)^*=\Hom_R(-, R)$ is exact. Therefore, $\cdots \to   F_1 \xrightarrow[]{\del_1} R \xrightarrow[]{[x]} R \xrightarrow[]{\del_1^*} F_1^* \xrightarrow[]{} \cdots \to F_{n-1}^* \to 0$ is a minimal projective resolution of $M_n\coloneqq \Omega_{n}(\sk)^*$, with $I^R_{n,1}(M_n)=(x)$. 
\end{example}

\section{ Deformation and the (Converse) Eisenbud--Shamash Construction}\label{sec:Deformation}
 
It is natural to ask whether there is any relationship between periodicity of ideals of minors over a noetherian local ring $R$ and $R[[x]]$ or $R/(x)$. The present section addresses this question, and we obtain some positive results for the same. One of the key ingredients we use is a converse to the Eisenbud--Shamash construction as described in \cite{BerghJorgensenMoore2020}. We reproduce it below as it will be used frequently in the remainder of the section.

\begin{construction}\cite[Section 3]{BerghJorgensenMoore2020}\label{InverseShamash}
Let $(R,\m)$ be a noetherian local ring, $x\in \m$ be a regular element, $R'=R/(x)$, and $M$ be a finitely generated $R'$-module.  {Given any free resolution $F'_\bullet$ of $M$ over $R'$ and for any lift $F_\bullet$ of $F'_\bullet$ to $R$, there exists a degree $-2$ endomorphism $\sigma$ of the complex $F_\bullet$ such that} a free resolution of $M$ as an $R$-module is given by
\[\dots\rightarrow F_n\oplus F_{n+1}\xrightarrow{\begin{bmatrix}
    \del_n^F & (-1)^{n+1}\sigma_{{n+1}}\\(-1)^nx&\del^F_{n+1}
\end{bmatrix}}F_{n-1}\oplus F_n\rightarrow\dots\rightarrow F_1\oplus F_2\xrightarrow{\begin{bmatrix}
    \del_1^F & \sigma_{{2}}\\-x&\del^F_{2}\end{bmatrix}} F_0\oplus F_1\xrightarrow{\begin{bmatrix}x&\del_1^F\end{bmatrix}}F_0\rightarrow 0.\]

    We denote this resulting resolution by $G(F'_\bullet)$.
\end{construction}

The above construction is minimal under additional assumptions, as established by the following:

\begin{lemma}
With the same assumptions as in \Cref{InverseShamash}, if, furthermore, $x\notin\m^2$ and $F'_\bullet$ is taken to be minimal, then {$G(F'_\bullet)$} is a minimal resolution of $M$ over $R$.
\end{lemma}
\begin{proof}
From \cite[Theorem 2.2.3]{Avramov1998}, the assumptions that $x\notin\m^2$ and that $F'_\bullet$ is a minimal resolution over $R'$ imply that $\beta_n^R(M) = \beta_n^{R'}(M) + \beta_{n-1}^{R'}(M)$. 
We observe, from the direct sum decomposition of $G_n$, that the ranks of free modules in the resolutions $F_\bullet'$ and {$G(F_\bullet')$} are related by the formula $\rank_{R'}(F_n') + \rank_{R'}(F_{n-1}') = \rank_R(G_n)$. Therefore, {$G(F'_\bullet)$} is a minimal free resolution of $M$ over~$R$.
\end{proof}

Let $(R, \m)$ be a noetherian local ring. Let $x \in \m$ be given. Then there exists a unique integer $i$ such that $x\in \m^i\setminus \m^{i+1}$. We let $x^*$ denote the element $x+\m^{i+1}$ of the associated graded ring $G_{\mathfrak m}(R):= \bigoplus_{n\geq 0} \m^n/\m^{n+1}$.
Following \cite{Sally79superregular}, when $x^*\in G_{\mathfrak m}(R)$ is regular, we call $x$ a \emph{super-regular element} in $R$. In a similar way, we call $x_1,\ldots, x_t \in \m$ a \emph{super-regular sequence} in $R$ if $x_1^*, \ldots, x_t^* \in G_{\mathfrak m}(R)$ is a regular sequence. It is easy to see that super-regular sequences are regular~sequences.

In the next theorem, we prove that when $x$ is a super-regular element, for a certain class of modules over $R/(x)$, the periodicity of ideals of minors can be lifted to periodicity over $R$. 
\begin{theorem}\label{thm:deformation-and-ideals-of-minors}
    Let $(R, \m_R)$ be a  noetherian local ring, $r$ be a positive integer, $x\in\m_R\setminus\m_R^2$ be super-regular, and $M$ be a finitely generated module over $R'=R/(x)$.  Suppose for each $1\leq s \leq r$, there exists $\ell_s\in \mathbb N$ such that $I_{n,s}^{R'}(M)=\m_{R'}^s$ for all $n\geq \ell_s$. 
%\\
 %   (a) If $\m_{R'}^r \neq 0$, then $I_{n,r}^R(M)=\m_R^r$ for all $n\geq \max\{\ell_1, \ldots, \ell_r\}$.
  %  \\
   % (b) 
\\   If there exists $N\in \mathbb N$ such that $\beta_n^{R'}(M)\geq r$ for all $n \geq N$, then $I_{n,r}^R(M)=\m_R^r$ for all $n\geq \max\{\ell_1, \ldots, \ell_r, N\}$.
\end{theorem}
\begin{proof}
       Let $F'_\bullet$ be a minimal resolution of $M$ over $R'$ and $G_\bullet=G(F'_\bullet)$ the associated minimal free resolution of $M$ over $R$.
    Let $x,y_2,\dots,y_e$ be a minimal set of generators of $\m_R$, so that the images of $y_2,\dots,y_e$ in $R'$ generate $\m_{R'}$.
    A set of generators of $\m_R^r$ is given by all elements of the form $x^{\alpha_1}y_2^{\alpha_2}\cdots y_e^{\alpha_e}$, where $\alpha_1+\dots +\alpha_e=r$. 

    Let $n\geq \max\{\ell_1,\ldots, \ell_r, N\}$, and $n\in \mathbb N$, let $A_n'$ denote the matrix representing $\partial_n^{F'}$. Since $n \geq N$, the matrix $A_n'$ has at least $r$ columns, i.e.,  $A_{n+1}'$ has at least $r$ rows.
    For each $n$, let $A_n=\partial_n^F$ be a matrix which lifts $A_n'$ to $R$, as in \Cref{InverseShamash}. 
    \par 
    Now, fix $\underline{\alpha}=(\alpha_1,\dots,\alpha_e)$ such that $\sum\limits_{t = 1}^e \alpha_t = r$. Then $s \coloneqq \sum\limits_{t = 2}^e \alpha_t  \leq r$. 
    Suppose that $A'_n$ has size $i \times j$.
    Since $I_{n,s}^{R'}(M)=\m_{R'}^s$, we must be able to obtain $y_2^{\alpha_2}\cdots y_e^{\alpha_e} \in R'$ as a sum of $s\times s$ minors of $A_n'$.
    We have an equality in $R'$ of the following form:
    
     \[y_2^{\alpha_2}\cdots y_e^{\alpha_e}=\sum_{(H,K)\in \binom{[i]}{s}\times \binom{[j]}{s}} a'_{H,K}\det(A'_{H,K}).\] 

    Here $a'_{H,K}\in R'$ are coefficients, $\det(A'_{H,K})$ denotes the $s\times s$ minor of $A'_n$ formed by choosing row indices $H$ and column indices $K$, and $\binom{[i]}{s}$ and $\binom{[j]}{s}$ denote the set of size $s$ subsets of $[i]=\{1,\dots, i\}$ and the set of size $s$ subsets of $[j]$, respectively. 

    For each $\underline{\alpha}$, there exists an element $z_{\underline{\alpha}}\in R$ and lifts $a_{H,K}$ of the coefficients $a'_{H,K}$ so that the following holds in $R$:
     \[y_2^{\alpha_2}\cdots y_e^{\alpha_e} + xz_{\underline{\alpha}} =\sum_{(H,K)\in \binom{[i]}{s}\times \binom{[j]}{s}} a_{H,K}\det(A_{H,K})\]

    where we note that the right-hand side is a sum of $s\times s$ minors of the matrix $A_n$.
   
Since $A_n$ has at least $r$ columns, there are at least $\alpha_1=(r-s)$ elements of $[j]\setminus K$ not yet chosen.
We would like to show that $x^{r-s}y_2^{\alpha_2}\cdots y_e^{\alpha_e}\in I_{n,r}^R(M)$. 
    To do this, we extend each $\det(A_{H,K})$ to an $r\times r$ minor of $\del_n^G$ by choosing a set $L$ of $\alpha_1$-many indices from $[j]\setminus K$, and for each index $\ell$ so chosen, choosing the corresponding row $\ell+i$ of $\del_n^G$. Observe that $\del_n^F$ and $\del^F_{n+1}$ are submatrices of $\del_n^G$ which occur in different blocks, and none of their rows or columns intersect. 
    \par Note that the $(\ell+i,\ell)^{th}$ entry of $\del^G_n$ is an $x$ in the lower left block of $\del_n^G$.
    Diagramatically, the resulting minor will be, after a permutation of the columns, the determinant of a matrix of the~form
    \[ {\begin{bmatrix}
        A_{H,L} & A_{H,K} \\
        x\cdot\id_{\alpha_1} & 0 \end{bmatrix}}\]
        By Laplacian expansion, the determinant of this matrix is $\pm x^{\alpha_1}\det(A_{H,K})$.

Repeating the above procedure for each minor and summing over the pairs $(H,K)$, we get that $x^{r-s}y_2^{\alpha_2}\cdots y_e^{\alpha_e} +x^{r-s+1}z_{\underline{\alpha}} \in I_{n,r}^R(M)$. 
Thus, to prove $x^{r-s}y_2^{\alpha_2}\cdots y_e^{\alpha_e} \in I_{n,r}^R(M)$, it suffices to show $x^{r-s+1}z_{\underline{\alpha}} \in I_{n,r}^R(M)$. 

Using induction on $s$, we establish a strengthening: that for every $1 \leq s \leq r$, we have $x^{r-s}\m_R^s \subseteq I_{n,r}^R(M)$.

For the base case $s=0$, running the above argument with no columns chosen from $A'_n$ shows that $x^r\in I_{n,r}^R(M)$.
Inductively, we assume that $x^{r-(s-1)}\m_R^{s-1}\subseteq I_{n,r}^R(M)$ and show the same holds for $s$.

As $x^{r-s}y_2^{\alpha_2}\cdots y_e^{\alpha_e}\in \m_R^r$, the expression  $x^{r-s}y_2^{\alpha_2}\cdots y_e^{\alpha_e}+x^{r-s+1}z_{\underline{\alpha}}\in I_{n,r}^R(M)\subseteq \m^r$ shows that $x^{r-s+1}z_{\underline{\alpha}}\in\m^r$. 
Since $x^*$ is a non-zero divisor in $G_{\mathfrak{m}_R}(R)$, this forces that $z_{\underline{\alpha}}$ belongs to $\m_R^{s-1}$.
Induction then gives that $x^{r-s+1}z_{\underline{\alpha}}\in I_{n,r}^R(M)$, and then subtracting this shows that $x^{r-s}y_2^{\alpha_2}\cdots y_e^{\alpha_e}\in I^R_{n,r}(M)$.
But as $\underline{\alpha}$ ranges over all choices with $\alpha_1=r-s$, we obtain a generating set of $x^{r-s}\m_R^s$, and hence $x^{r-s}\m_R^s\subseteq I_{n,r}^R(M)$, which completes the proof.
\end{proof}

Observe that if $R'$ is a noetherian local ring and $R=R'[[x]]$, then the element $x\in R$ is super-regular. Hence, as an immediate consequence of the above theorem, we obtain the following.

\begin{corollary}\label{cor:deformation-and-ideals-of-minors-power-series}
    Let $(R',\m_{R'})$ be a  noetherian local ring, $R=R'[[x]]$, $r$ be a positive integer, and $M$ be a finitely generated $R'$-module.  Suppose for each $1\leq s \leq r$, there exists $\ell_s\in \mathbb N$ such that $I_{n,s}^{R'}(M)=\m_{R'}^s$ for all $n\geq \ell_s$. 
\\If there exists $N\in \mathbb N$ such that $\beta_n^{R'}(M)\geq r$ for all $n \geq N$, then $I_{n,r}^R(M)=\m_R^r$ for all $n\geq \max\{\ell_1, \ldots, \ell_r, N\}$.
\end{corollary}
% \begin{proof}
%     \Cref{thm:deformation-and-ideals-of-minors} applies since $x\in R$ is a super-regular element.
% \end{proof}

%The following result provides another instance where one can lift the periodicity of ideals of minors from $R/(x)$ to $R$. 

% \begin{corollary}\label{thm:deformation-and-ideals-of-minors-mR-non-zero}
%     Let $(R, \m_R)$ be a  noetherian local ring, $r$ be a positive integer, $x\in \m\setminus\m^2$ be a super-regular element, and $M$ be a finitely generated module over $R'=R/(x)$.  Suppose for each $1\leq s \leq r$, there exists $\ell_s\in \mathbb N$ such that $I_{n,s}^{R'}(M)=\m_{R'}^s$ for all $n\geq \ell_s$. 
% If $\m_{R'}^r \neq 0$, then $I_{n,r}^R(M)=\m_R^r$ for all $n\geq \max\{\ell_1, \ldots, \ell_r\}$.
% \end{corollary}
% \begin{proof}
%     Since for all $n \geq l_r$, we have $I_{n,r}^{R'}(M) = \m_{R'}^r \neq 0$, the matrix $\partial_n^{F'}$ has at least $r$ columns. In other words, $\beta_n^{R'}(M) \geq r$ for all $n \geq l_r$. Taking $N=\ell_r$, \Cref{thm:deformation-and-ideals-of-minors} applies.
% \end{proof}

In the remainder of this section, we obtain generalizations of \Cref{t:FiberProduct} and \Cref{thm:main-theorem-stretched-Gorenstein} to certain quasi-fiber products and stretched Gorenstein rings of arbitrary dimension. Here, we say that a ring $R$ is a \emph{quasi-fiber product} if there exists an $R$-regular sequence $x_1,\ldots,x_t$ in  $R$ such that $R/(x_1,\ldots,x_t)$ is a fiber product ring. We begin with the following lemmas on the structure of quasi-fiber products. 
\begin{lemma}\label{lemma:FP-non-regular}
    Let $(R,\m)$ be a local ring and $x\in \m$ be an $R$-regular element.  If $R$ is non-regular and $R'=R/(x)$ is  a quasi-fiber product ring, then $x\notin\m^2$.
    If instead $R$ is a quasi-fiber product ring and $x_1,\ldots,x_t$ is an $R$-regular sequence such that $R/(x_1,\dots, x_t)$ is an associated fiber product ring, then the following are true:

    \begin{enumerate}[\rm(1)]
        \item If $R$ is not regular, then $x_i\notin\m^2$ for all $i$.

        \item If $R$ is regular, then none of $x_1,\dots,x_t$ are  in $\m^2$ except possibly at most one. 
    \end{enumerate}

\end{lemma}
\begin{proof}
    By \cite[Corollary 6.5]{quasidec}, $R'$ is Tor-friendly in the terminology of \cite{persistence}. Passing to the completion, we see that $\widehat{R'}\cong \widehat R/x\widehat R$, the image of $x$ in $\widehat R$ is regular, and $\widehat R$ is not regular.  We then get that $x\widehat R\nsubseteq \m^2\widehat R$ by \cite[Proposition 2.8(2)]{persistence}. Thus, $x\notin \mathfrak m^2$. 

    Now for the last claim, (1) follows from the above since for each $i$, $R_i \coloneqq R/(x_i)$ is a quasi-fiber product ring.
    For (2), without loss of generality, we may assume $x_1\in \m^2$. Then, $S \coloneqq R/(x_1)$ is not regular, and we can repeatedly apply the above (with $R'=R_i \coloneqq S/(x_i)$ for each $i\ge 2$) to finish the claim.
\end{proof}

\begin{lemma}\label{quasiFib-edimlem} Let $R$ be a quasi-fiber product ring which is not Golod. Let $\underline{x}=x_1,\ldots, x_t$ be an $R$-regular sequence such that $R'=R/( \underline{x})$ is an associated fiber product ring. Then, $\edim(R')\ge 3$.
\end{lemma}

\begin{proof} By \Cref{lemma:FP-non-regular}, we see that $x_1, \ldots, x_t \in \m_R \setminus\m_R^2$ (note that $R$ is not regular since regular rings are Golod). 
Since $R$ is not Golod, the same holds for $R'$ by \cite[Proposition 5.2.4]{Avramov1998}.
$R'$ also cannot be a complete intersection, since this would imply that it is a hypersurface (see \cite[Corollary 2.7]{bog}), and hypersurface rings are Golod (see \cite[Proposition 5.2.5]{Avramov1998}).
Therefore $\embdim(R')\geq 3$  by \cite[5.1, Proposition 5.3.4]{Avramov1998}.
\end{proof}

\Cref{thm:deformation-and-ideals-of-minors} allows us to extend \Cref{t:FiberProduct} to a certain class of modules over certain quasi-fiber products, as given in the corollary below.

\begin{corollary}\label{cor:QFP}
Let $R$ be a  ring which is not Golod. Assume there exists a super-regular sequence $\underline{x}=x_1,\ldots, x_t$ in $R$  such that $R'=R/( \underline{x})$ is a fiber product ring.
If $M$ is a finitely generated $R'$-module of infinite projective dimension, then $I_{n,r}^R(M)=\m_R^r$ for all $n \gg0$.
\end{corollary}
\begin{proof}
By \Cref{quasiFib-edimlem} we have $\embdim(R')\geq 3$, so \Cref{t:FiberProduct} implies that $I_{n,r}^{R'}(M)=\m_{R'}^r$ for all $n\gg 0$. Furthermore, from \Cref{rem:betti-nos-over-FP}, we see that for $n\gg0$, we have $\beta^{R'}_n(M)\geq  r$. Thus, using \Cref{thm:deformation-and-ideals-of-minors} inductively, the result follows.  
\end{proof}

Note that $R=R'[[x_1,\ldots,x_t]]$, where $R'$ is a fiber product of embedding dimension at least $3$ is a special case of the above corollary.

When $(R,\m)$ is Cohen--Macaulay of minimal multiplicity, any quotient of $R$ by an ideal generated by a minimal reduction of $\m$ is a fiber product (c.f.~\cite[Example 4.7]{quasidec}). Thus, when a minimal reduction of $\m$ exists for such rings (e.g. when $R/\mathfrak m$ is infinite), the ring $R$ is a quasi-fiber product. Moreover, in this case, since $G_{\mathfrak m}(R)$ is Cohen--Macaulay, every minimal reduction is a super-regular sequence (see \cite[Theorems 1 and 2]{Sa77}). Thus, arguing as in the above corollary, we get the following consequence.  

\begin{corollary}
    Let $(R,\m)$ be a Cohen--Macaulay ring of minimal multiplicity with infinite residue field and $\embdim(R)-\dim(R)\geq 3$. Suppose  $\underline{x}=x_1, \ldots, x_d$ is a minimal reduction of $\m$ and $R'= R/ (\underline{x})$. Then for every non-free $R'$-module $M$, we have $I_{n,r}^R(M)=\m_R^r$ for all $n \gg0$.
\end{corollary}

Similar to the case of quasi-fiber products above, \Cref{thm:main-theorem-stretched-Gorenstein} can be extended to a certain class of modules over stretched Gorenstein rings of arbitrary dimension. 

\begin{corollary}\label{cor:SGR}
    Let $(R, \m, \mathsf k)$ be a local ring of dimension $d$ which is neither Golod nor a complete intersection. Suppose $\ch(\mathsf k)\neq 2$. Let $\underline{x}=x_1,\ldots, x_d$ be a minimal reduction of $\m$ such that $R'=R/( \underline{x})$ is an artinian stretched Gorenstein ring, and $\underline{x}$ be a super-regular sequence in $R$. Then given any finitely generated non-free $R'$-module $M$, we have $I_{n,r}^R(M)=\m_R^r$ for all $n \gg 0$.
\end{corollary}
\begin{proof}
An argument similar to the proof of \Cref{quasiFib-edimlem} shows $\edim(R')\ge 3$, so \Cref{thm:main-theorem-stretched-Gorenstein} implies that $I_{n,r}^{R'}(M)=\m_{R'}^r$ for all $n\gg 0$. By \Cref{rem:betti-no-over-artinian-stretched}, we see that for $n \gg 0$, we have $\beta^{R'}_n(M)\geq r$. Using \Cref{thm:deformation-and-ideals-of-minors} inductively, the result follows. 
\end{proof}

 As a special case of the above corollary, we have the following.
\begin{corollary}
 Let $(R, \m, \mathsf k)$ be a local ring which is neither Golod nor a complete intersection, $\ch(\mathsf k)\neq 2$,  $\mathsf k$ is infinite, $R$ is stretched Gorenstein and  $G_{\mathfrak m}(R)$ is Cohen--Macaulay. Then, for any minimal reduction $\underline{x}=x_1,\ldots, x_d$  of $\m$ such that $R'=R/( \underline{x})$ is an artinian stretched Gorenstein ring, if $M $ is a  finitely generated non-free $R'$-module, then we have $I_{n,r}^R(M)=\m_R^r$ for all $n \gg 0$.
 \end{corollary}

\end{document}